\documentclass[12pt]{article}
\usepackage{amsmath,amsthm,amssymb,amsfonts}
\usepackage{enumerate}
\usepackage{mathrsfs}
\usepackage[cmtip,all]{xy}

\normalsize

\newtheoremstyle{theoremstyle}
  {10pt}      %  Space above
  {5pt}       %  Space below
  {\itshape}  %  Body font
  {}          %  Indent amount (empty = no indent, \parindent = para indent)
  {\bfseries} %  Thm head font
  {:}         %  Punctuation after thm head
  {.5em}      %  Space after thm head: " " = normal interword space;
  {}          %  Thm head spec (can be left empty, meaning `normal')

\newtheoremstyle{examplestyle}
  {10pt}      %  Space above
  {5pt}       %  Space below
  {}          %  Body font
  {}          %  Indent amount (empty = no indent, \parindent = para indent)
  {\bfseries} %  Thm head font
  {:}         %  Punctuation after thm head
  {.5em}      %  Space after thm head: " " = normal interword space;
  {}          %  Thm head spec (can be left empty, meaning `normal')
  
\theoremstyle{theoremstyle}
\newtheorem{theorem}{Theorem}[section]
\newtheorem*{theorem*}{Theorem}
\newtheorem{lemma}[theorem]{Lemma}
\newtheorem{proposition}[theorem]{Proposition}
\newtheorem*{proposition*}{Proposition}
\newtheorem{corollary}[theorem]{Corollary}
\newtheorem*{corollary*}{Corollary}

\newtheorem{example}[theorem]{Example}
\newtheorem{definition}[theorem]{Definition}
\newtheorem{definition*}{Definition}
\newtheorem{remark}[theorem]{Remark}
\newtheorem{remark*}{Remark}

\newcommand{\G}{\mathbf{G}_{\mathfrak{m}}}
\newcommand{\A}{\mathbf{A}}
\newcommand{\B}{\mathsf{B}}

\newcommand{\PP}{\mathbf{P}}

\newcommand{\CC}{\mathbf{C}}

\newcommand{\BG}{\mathsf{B}\mathbf{G}_{\mathfrak{m}}}

\newcommand{\KGL}{\mathsf{KGL}}

\newcommand{\id}{\mathsf{id}}

\newcommand{\EE}{\mathsf{E}}

\newcommand{\Z}{\mathbf{Z}}

\newcommand{\sh}{\mathsf{sh}}
\newcommand{\colim}{\mathrm{colim}}

\renewcommand{\smash}{\wedge}

\newcommand{\sign}{\mathrm{sign}}

\title{{\bf Motivic strict ring models for $K$-theory}}
\author{Oliver R{\"o}ndigs, Markus Spitzweck, Paul Arne {\O}stv{\ae}r}
\date{July 23, 2009}
\begin{document}
\maketitle
\begin{abstract}
It is shown that the $K$-theory of every noetherian base scheme of finite Krull dimension is represented by a strict ring object in 
the setting of motivic stable homotopy theory.
The adjective `strict' is used to distinguish between the type of ring structure we construct and one which is valid only up to homotopy. 
Both the categories of motivic functors and motivic symmetric spectra furnish convenient frameworks for constructing the ring models.
Analogous topological results follow by running the same type of arguments as in the motivic setting.
\end{abstract}
\tableofcontents
\newpage

\section{Introduction}
\label{section:introduction}
Motivic homotopy theory can be viewed as an expansion of classical homotopy theory to an algebro-geometric setting. 
This has enabled the introduction of homotopy theoretic techniques in the study of generalized ring (co)homology theories for schemes,
and as in classical algebra one studies these via modules and algebras.
From this perspective, 
motives are simply modules over the motivic Eilenberg-MacLane ring spectrum \cite{RO1}, \cite{RO2}.
The main purpose of this paper is to show that the $K$-theory of every noetherian base scheme of finite Krull dimension acquires strict 
ring object models in motivic homotopy theory, 
and thereby pave the way towards a classification of modules over $K$-theory.
An example of a base scheme of particular interest is the integers.
Working with some flabby smash product which only become associative, commutative and unital after passage to the motivic stable 
homotopy category is inadequate for our purposes.
The whole paper is therefore couched in terms of motivic functors \cite{DROstable} and motivic symmetric spectra \cite{Jardinestable}.
Throughout the paper the term `$K$-theory' is short for homotopy algebraic $K$-theory.
\vspace{0.2in}

Fix a noetherian base scheme $S$ of finite Krull dimension with multiplicative group scheme $\G$. 
Denote by $\KGL$ the ordinary motivic spectrum representing $K$-theory \cite{Voevodskystable}.
As shown in \cite{SO}, 
see also \cite{GS} and \cite{NSOfield},
inverting a homotopy class $\beta\in \pi_{2,1}\Sigma^{\infty}{\BG}_{+}$ in the motivic suspension spectrum of the classifying space 
of the multiplicative group scheme yields a natural isomorphism in the motivic stable homotopy category
\begin{equation*}
\xymatrix{
\Sigma^{\infty}{\BG}_{+}[\beta^{-1}]
\ar[r]^-{\cong} &
\KGL. }
\end{equation*}
We shall turn the Bott inverted model for $K$-theory into a commutative monoid $\KGL^{\beta}$ in the category of motivic 
symmetric spectra.
To start with, 
the multiplicative structure of $\G$ induces a commutative monoid structure on the motivic symmetric suspension spectrum 
of the classifying space ${\BG}_{+}$.
A far more involved analysis dealing with an actual map rather than some homotopy class allows us to define $\KGL^{\beta}$ 
and eventually verify that it is a commutative monoid with the same homotopy type as $K$-theory.
Several of the main techniques employed in the proof are of interest in their own right, 
and can be traced back to constructions for symmetric spectra, cf.~\cite{Schwede:book}.
It also turns out that there exists a strict ring model for $K$-theory in the category of motivic functors.
Here the motivic functor model is constructed in a leisurely way by transporting $\KGL^{\beta}$ via the strict symmetric monoidal 
functor relating motivic symmetric spectra to motivic functors. 
\vspace{0.2in}
 
When suitably adopted the motivic argument works also in topological categories.   
The topological strict ring models appear to be new, even in the case of symmetric spectra.
\vspace{0.2in}

The Bott element considered by Voevodsky in \cite{Voevodskystable} is obtained from the virtual vector bundle
\begin{equation*}
[\mathcal{O}_{\PP^{1}}]
-
[\mathcal{O}_{\PP^{1}}(-1)].       
\end{equation*}
A key step in the construction of $\KGL^{\beta}$ is to interpret the same element,
viewed in the pointed motivic unstable homotopy category,  
as an actual map between motivic spaces.
In order to make this part precise we shall use a lax symmetric monoidal fibrant replacement functor for pointed motivic spaces.
Fibrancy is a constant source for extra fun in abstract homotopy theory.  
The problem resolved in this paper is no exception in that respect.
It is also worthwhile to emphasize the intriguing fact that $\beta$ does not play a role in the definition of the multiplicative
structure of $\KGL^{\beta}$.
However, 
the Bott element enters in the definition of the unit map ${\bf 1}\to\KGL^{\beta}$, 
which is part of the monoid structure, 
and in the structure maps.
In fact, 
up to some fibrant replacement, 
$\KGL^{\beta}$ is constructed fairly directly from $\Sigma^{\infty}{\BG}_{+}$ by intertwining a map representing $\beta$ 
with the structure maps.
On the level of homotopy groups this type of intertwining has the effect of inverting the Bott element.
As a result, 
we obtain the desired homotopy type.
\vspace{0.2in}

In \cite{PPR1} it is shown that under a certain normalization assumption the ring structure on $\KGL$ in the motivic stable homotopy 
category is unique over the ring of integers $\Z$. 
For any base scheme $S$ the multiplicative structure pulls back to give a distinguished monoidal structure on $\KGL$.
We show the multiplicative structures on $\KGL^{\beta}$ and $\KGL$ coincide in the motivic stable homotopy category.
The proof of this result is not formal.
A key input is that for $K$-theory there exists no nontrivial phantom maps.
In turn, 
this is a consequence of Landweber exactness in motivic homotopy theory \cite{NSO}.
\vspace{0.2in}

In \cite{GS} the setup of $\infty$-categories is used to note the existence of an $E_{\infty}$ or coherently homotopy 
commutative structure on $K$-theory.
Work in progress suggests there exists a unique such structure. 
However, 
with the construction of $\KGL^{\beta}$ in hand one has strict models for $K$-theory.
And the strictness of a model has the pleasing consequence that it is amenable to a simpler homotopical study.
This has been the subject to much work dealing with topological $K$-theories.

\section{A strict model}
\label{section:thestrictmodels}
The main focus of this section is the construction of a strict model for $K$-theory in the category of motivic symmetric spectra.
Throughout we use the `closed' motivic model structure in \cite{PPR1} with a view towards realization functors.
An extensive background in motivic stable homotopy theory is not assumed.
\vspace{0.1in}

The classifying space $\BG$ has terms $\G^{\times n}$ for $n\geq 0$ with the convention that its zeroth term is a point.
Its face and degeneracy maps, 
which are defined in a standard way using diagonals and products, 
allow to consider $\BG$ as a motivic space
(that is, a simplicial presheaf on the Nisnevich site of the base scheme $S$).
\vspace{0.1in}

Throughout we use the following standard notation:
Let $S^{2,1}$ denote the motivic sphere defined as the smash product of the simplicial circle $S^{1,0}=\Delta^1/\partial\Delta^1$ with $\G$ pointed
by its one-section.
For $n\geq 2$ we set $S^{2n,n} = S^{2n-2,n-1}\smash S^{2,1}$. 
When forming motivic spectra, we shall for consistency with \cite{Schwede:book} be smashing with $S^{2,1}$ on the right. 
\vspace{0.1in}

In the introduction it was recalled that the Bott element is a homotopy class 
\begin{equation*}
\beta\in \pi_{2,1}\Sigma^{\infty}{\BG}_{+}. 
\end{equation*}
As such, 
it is represented by a map of pointed motivic spaces 
\begin{equation*}
S^{2n+2,n+1}\to ({\BG}_+ \smash S^{2n,n} )^{\textrm{fib}}
\end{equation*}
for some $n$, 
where $(-)^{\textrm{fib}}$ denotes a fibrant replacement functor. 
The construction we give of $\KGL^{\beta}$ works for any such representative provided the fibrant replacement 
functor is lax symmetric monoidal.  
By Lemma~\ref{lem:fibrant-repl} we may choose a fibrant replacement functor with the stated properties.
As shown in the  following,
the situation at hand allows for an explicit construction of a map 
\begin{equation*}
S^{4,2}\to ({\BG}_+ \smash S^{2,1} )^{\textrm{fib}}
\end{equation*}
that represents the Bott element.
\vspace{0.1in}

Before proceeding with the construction of the strict models we discuss fibrancy of the multiplicative group scheme 
and its classifying space, 
pertaining to the discussion of a fibrant replacement functor in the above.

\begin{example}
The classifying space of the multiplicative group scheme is sectionwise fibrant because it takes values in 
simplicial abelian groups.
When $S$ is regular then $\G$ is fibrant.
However, 
as the following discussion shows, 
$\BG$ is not fibrant.

The standard open covering of the projective line by affine lines yields an elementary distinguished square:
\begin{equation}
\label{equation:eds}
\xymatrix{ 
\G      \ar[r] \ar[d] & \A^{1} \ar[d] \\
\A^{1}  \ar[r]        & \PP^{1} }
\end{equation}
Let $P$ denote the homotopy pullback of the diagram
\begin{equation*}
\xymatrix{ 
\BG(\A^{1}) \ar[r] & \BG(\G) & \BG(\A^{1}) \ar[l] }
\end{equation*}
obtained by applying $\BG$ to $(\ref{equation:eds})$.
Then the homotopy fiber of the map $P\to\BG(\A^{1})$ is weakly equivalent to the homotopy fiber $F$ of $\BG(\A^{1})\to\BG(\G)$.
For a ring $R$, 
let $R^{\times}$ denote its multiplicative group of units.
With these definitions there exist induced exact sequences of homotopy groups
\begin{equation*}
\xymatrix{ 
0\ar[r] & 
\pi_{1}P\ar[r] & 
\mathcal{O}_{S}^{\times}=\mathcal{O}_{S}[t^{-1}]^{\times}\ar[r] &
\pi_{0}F\ar[r] & 
\pi_{0}P\ar[r] & 
0, }
\end{equation*}
and  
\begin{equation*}
\xymatrix{ 
0\ar[r] & 
\mathcal{O}_{S}^{\times}=\mathcal{O}_{S}[t]^{\times}\ar[r] &
\mathcal{O}_{S}[t,t^{-1}]^{\times}\ar[r] &
\pi_{0}F\ar[r] & 
0. }
\end{equation*}
(Here $t$ is an indeterminate.)
Hence $\mathcal{O}_{S}[t^{-1}]^{\times}\to\pi_{0}F$ cannot be surjective.
It follows that $P$ is not connected and, 
in particular, 
not weakly equivalent to $\BG(\PP^{1})=\B\mathcal{O}_{S}^{\times}$.
This shows that $\BG$ does not satisfy the Nisnevich fibrancy condition, see 
\cite{DROstable}, \cite{Jardinestable}, \cite{MV}.
\end{example}
\vspace{0.1in}

We write
\[ 
i\colon S^{1,0}\smash \G\to \BG 
\]
for the inclusion of the $1$-skeleton $\G$ into the classifying space $\BG$.
Let 
\[ 
c\colon S^{1,0}\smash \G\to \BG 
\]
denote the constant map. 
Via the motivic weak equivalences 
$S^{1,0}\smash \G\simeq\PP^{1}$ 
and  
$\BG\simeq\PP^{\infty}$
the map $i$ can be identified in the pointed motivic unstable homotopy category with the inclusion 
\[
\PP^{1}\rightarrow\PP^{\infty}.
\]
In homogeneous coordinates the inclusion map is given by 
\[
[x:y]\mapsto [x:y:0:\cdots].
\]
Similarly, 
the map $c$ coincides with the canonical composite map 
\[
\PP^{1}\rightarrow S\rightarrow\PP^{\infty}
\]
given by 
\[
[x:y]\mapsto [1:0:\cdots].
\]
\vspace{0.1in}

Adding a disjoint base point to the classifying space of $\G$ yields pointed maps 
\[ 
i_+,c_+ \colon S^{2,1} \to {\BG}_+ 
\] 
for the base point of ${\BG}$.
Now in order to move the base point in ${\BG}_+$ we take the unreduced suspension of both these maps. 
Recall the unreduced suspension of a motivic space $A$ is defined as the pushout
\[ 
S(A)
=  
A\times \Delta^1\cup_{A\times \partial \Delta^1} \partial \Delta^1. 
\]
One can view it as a pointed motivic space by the image of $0\in \partial \Delta^1$.
With this definition, 
the unreduced suspensions of the maps $i_+$ and $c_+$ are pointed with respect to the image of $(+,0)$ in their target.
If $A$ is pointed, the canonical map $q\colon S(A)\to \Sigma A$ to the reduced suspension is a weak equivalence. 
Hence there exists a map of pointed motivic spaces 
\[ 
S^{4,2} 
\to 
(S(S^{2,1})\smash \G)^{\mathrm{fib}} 
\]
lifting the inverse of the map $q\smash \G$ in the pointed motivic unstable homotopy category.
By composing we end up with the two pointed maps
\[ 
i^\beta, 
c^\beta
\colon 
S^{4,2} 
\xrightarrow{\sim}
(S(S^{2,1})\smash \G)^{\mathrm{fib}} 
\rightrightarrows
(S({\BG}_+) \smash \G )^{\mathrm{fib}}
\xrightarrow{\sim}
({\BG}_+ \smash S^{2,1} )^{\mathrm{fib}}.
\]

As a first approximation of the Bott element $\beta$ we consider the analog of the virtual vector bundle
$[\mathcal{O}_{\PP^{1}}]-[\mathcal{O}_{\PP^{1}}(-1)]$ in the pointed motivic unstable homotopy category
\begin{equation}
\label{formaldifference}
c^\beta - i^\beta
\colon 
S^{4,2} 
\to 
({\BG}_+ \smash S^{2,1})^{\mathrm{fib}} .
\end{equation}
In order to form the difference map we use that $S^{4,2}$ is a (two-fold) simplicial suspension, 
and therefore a cogroup object in the pointed motivic unstable homotopy category. 
Note, 
however, 
that $c^\beta$ represents the trivial map because it factors through the base point. 
By appealing to the motivic model structure it follows that (\ref{formaldifference}) lifts to a `strict' motivic Bott map 
\[ 
S^{4,2} 
\to 
({\BG}_+ \smash S^{2,1})^{\mathrm{fib}}.
\]
between pointed motivic spaces.
Here we use that the motivic sphere is cofibrant in the closed motivic model structure.
\vspace{0.1in}

The fibrancy caveat above requires us to replace the suspension object
$\Sigma^\infty {\BG}_+$ with a levelwise fibrant motivic spectrum. 
An arbitrary such replacement need not preserve commutative monoids.
The following lemma will therefore be of relevance later in the paper. 
\begin{lemma}
\label{lem:fibrant-repl}
There exists a lax symmetric monoidal fibrant replacement functor $\mathrm{Id}\to F$ on the category of pointed motivic spaces.
\end{lemma}
\begin{proof}
The straightforward simplicial presheaf analog of \cite[Theorem 2.1.66]{MV} provides a lax symmetric monoidal fibrant 
replacement functor $\textrm{Ex}^{\infty}$ for the local model structure on any site of finite type. 
Moreover, 
the singular endofunctor $\mathrm{Sing}_{\ast}$ \cite{MV} constructed by means of the standard cosimplicial $S$-scheme 
$\Delta^{\bullet}_{\A^{1}}$ with terms
\begin{equation*}
\Delta^{n}_{\A^{1}}=
S\times_{\textrm{Spec}(\Z)}\textrm{Spec}(\Z[x_0,\dots,x_n]/x_0+\dots+x_{i}-1)
\end{equation*}
is strict symmetric monoidal, 
because it commutes with limits and colimits. 
Thus the lemma follows by using the iterated construction 
\begin{equation*}
\textrm{Ex}^{\infty}\circ (\textrm{Ex}^{\infty}\circ \mathrm{Sing}_{\ast})^{\omega} \circ \textrm{Ex}^{\infty}
\end{equation*}
as the fibrant replacement functor \cite[Lemma 3.2.6]{MV}.
In this definition $\omega$ denotes the cardinality of the natural numbers.
\end{proof}

\begin{corollary}
\label{cor:inf-proj-commutative}
The motivic symmetric spectrum $F(\Sigma^\infty {\BG}_+)$ obtained by applying the functor $F$ levelwise to  $\Sigma^\infty {\BG}_+$
is a commutative monoid.
\end{corollary}
\begin{proof}
The assertion follows immediately by combining Lemma \ref{lem:fibrant-repl} and the fact that $\G$ is a commutative group scheme.
\end{proof}

\begin{corollary}
\label{cor:bott-central}
There exists a motivic Bott map between pointed motivic spaces
\[ 
b\colon S^{4,2}\to F({\BG}_+ \smash S^{2,1} ) 
\]
that represents the difference map 
\[
c^\beta- i^\beta
\]
in the pointed motivic unstable homotopy category. 
The map $b$ is central in the sense that the diagram
\[
\xymatrix{ 
F({\BG}_+)^{n}  \smash S^{4,2} \ar[r]^-{\id \smash b} \ar[d]_{\mathrm{twist}}^{\cong} & 
F({\BG}_+)^{n}  \smash F({\BG}_+)^{1} \ar[r]^-{\mu_{n,1}} & 
F({\BG}_+)^{n+1} \ar[d]^{\chi_{n,1}} \\
S^{4,2}\smash F({\BG}_+)^{n}  \ar[r]^-{b\smash \id} & 
F({\BG}_+)^{1} \smash F({\BG}_+)^{n}  \ar[r]^-{\mu_{1,n}} &
F({\BG}_+)^{1+n} } 
\]
commutes.
Here $F({\BG}_+)^{k}$ is short for $F({\BG}_+ \smash S^{2k,k})$  and $\chi_{n,1}$ denotes the cyclic permutation $(1,2,\dots,n,n+1)$.
\end{corollary}
\begin{proof}
This is immediate from Corollary \ref{cor:inf-proj-commutative} and the commutativity of $\G$.
\end{proof}

With these preliminary results in hand we are ready to construct a strict ring model for $K$-theory.
In the following we shall adopt constructions for symmetric spectra given in Schwede's manuscript \cite{Schwede:book} 
to the setting of motivic symmetric spectra.
Let $\Omega^{2n,n}$ denote the right adjoint of the suspension functor $-\smash S^{2n,n}$ on pointed motivic spaces.
\vspace{0.2in}

Define $\KGL^{\beta}$ to be the motivic symmetric spectrum with constituent spaces
\[ 
\KGL^{\beta}_n= 
\Omega^{4n,2n}F({\BG}_+ \smash S^{4n,2n} ). 
\]
The group $\Sigma_{n}$ acts on $S^{4n,2n}$ and therefore also on $F({\BG}_+ \smash S^{4n,2n} )$ 
via restriction along the diagonal embedding
\begin{equation*}
\Delta_{n}
\colon
\Sigma_{n}
\to
\Sigma_{2n}
\end{equation*}
defined for $1\leq j\leq 2$ and $1\leq i\leq n$ by setting
\begin{equation*}
\Delta_{n}(\sigma)(j+2(i-1))
=
j+2(\sigma(i)-1).
\end{equation*}
Now the $\Sigma_{n}$-action on the $(4n,2n)$-loop space $\KGL^{\beta}_n$ is defined by conjugation.
That is, 
for elements $\sigma\in\Sigma_{n}$ and $\phi\in\KGL^{\beta}_n$ define 
\[ 
\sigma\cdot\phi(-)=\sigma(\phi(\sigma^{-1}(-))).
\]
In this definition, taking sections is implicit in the notation. 
\vspace{0.2in}

Let 
\[ 
\mu_{m,n}\colon
F({\BG}_+ \smash S^{2m,m} )
\smash
F({\BG}_+ \smash S^{2n,n} ) 
\to
F({\BG}_+ \smash S^{2(m+n),m+n} )
\]
denote the maps comprising the multiplicative part of the monoid structure on 
\[ 
F(\Sigma^\infty {\BG}_+).
\]

Define the multiplication map 
\begin{equation}
\label{multiplication}
\KGL^{\beta}_m\smash \KGL^{\beta}_n \to \KGL^{\beta}_{m+n}
\end{equation}
by
\[ 
f\smash g \mapsto \mu_{2m,2n}\circ (f\smash g).
\]

The multiplication map (\ref{multiplication}) is strictly associative on account of the strict associativity of the smash product 
and the multiplicative structure on $F(\Sigma^\infty {\BG}_+)$.
\vspace{0.1in}

Moreover, 
(\ref{multiplication}) is $\Sigma_{m}\times\Sigma_{n}$-equivariant due to the equivariance of the multiplicative structure on 
$F(\Sigma^\infty {\BG}_+)$ and the compatibility relation
\[
\Delta_{m}(\sigma)\times\Delta_{n}(\sigma')=\Delta_{m+n}(\sigma\times\sigma')
\]
for the diagonal embeddings $\Delta_{k}\colon\Sigma_{k}\to\Sigma_{2k}$
(in our cases of interest $k=m,n,m+n$).
\vspace{0.1in}

Let
\[ 
F( {\BG}_+ \smash S^{2n,n} )
\to 
\KGL^{\beta}_n= 
\Omega^{4n,2n} F( {\BG}_+ \smash S^{4n,2n} ) 
\]
be the adjoint of the composite map of 
\[ 
F( {\BG}_+ \smash S^{2n,n} )\smash S^{4n,2n} \! 
\xrightarrow{\id\smash b^{\smash n}}  
\!\!F({\BG}_+ \smash S^{2n,n} ) \smash  F({\BG}_+ \smash S^{2n,n} ) 
\xrightarrow{\mu_{n,n}}  
\!\!F( {\BG}_+ \smash S^{4n,2n} )
\]
and 
\[ 
F( {\BG}_+ \smash S^{4n,2n} )\to
F( {\BG}_+ \smash S^{4n,2n} )
\]
given by the permutation $\sigma\in\Sigma_{2n}$ defined by 
\begin{equation*}
\sigma(i)=
\begin{cases}
1+2(i-1) & 1\leq i\leq n \\
2+2(k-1) & i=n+k, 1\leq k\leq n. 
\end{cases}
\end{equation*}

These maps assemble into a map of motivic symmetric ring spectra
\[ 
F(\Sigma^{\infty}{\BG}_+)
\to 
\KGL^{\beta}.
\]
In fact the structure maps 
\begin{equation*}
\KGL^{\beta}_{n}\smash S^{2,1} \rightarrow\KGL^{\beta}_{n+1}
\end{equation*}
and the unit map of $\KGL^{\beta}$ are obtained from the above and the unit map of $F(\Sigma^\infty {\BG}_+)$.

\begin{lemma}
\label{lem:inv-ring-spectrum}
The motivic symmetric spectrum $\KGL^{\beta}$ is a commutative monoid.
\end{lemma}
\begin{proof}
The equation for commutativity  
\begin{equation*}
\mu_{2m,2n}\circ (f\smash g)=
\mu_{2n,2m}\circ (g\smash f)
\end{equation*}
holds because $\G$ is a commutative group scheme.
\end{proof}

\section{The homotopy type}
\label{section:proofs}

In this section we finish the proof of our main result and elaborate further on some closely related results.
First we prepare for the comparison of $\KGL^{\beta}$ with the homotopy colimit of the Bott tower introduced in \cite{SO}. 
\vspace{0.1in}

Let $\sh(-)$ denote the shifted motivic symmetric spectrum functor defined by 
\begin{equation*}
\sh(\EE)_{n}
=
\EE_{1+n}.
\end{equation*}
Its  structure maps are induced from the ones for $\EE$ by reindexing.
The $\Sigma_{n}$-action on the $n$th term of $\sh(\EE)$ is determined by the injection 
$(1\times -)\colon\Sigma_{n}\to\Sigma_{1+n}$ given by 
\begin{equation*}
(1\times\sigma)(i)
=
\begin{cases}
1 & i=1 \\
\sigma(i-1)+1 & i\neq 1.
\end{cases}
\end{equation*}
For our purposes the main application of the shift functor is to introduce the notion of a semistable motivic symmetric spectrum.
\vspace{0.1in}

There exists a natural map 
\begin{equation}
\label{naturalshiftmap}
\phi(\EE)
\colon
S^{2,1}\smash \EE \to \sh(\EE).
\end{equation}
In level $n$ it is defined as the composite map
\[ 
S^{2,1}\smash \EE_{n} 
\xrightarrow{\cong} 
\EE_{n} \smash S^{2,1}
\to
\EE_{n+1} 
\to
\EE_{1+n}
\]
of the twist isomorphism, the $n$th structure map of $\EE$ and the cyclic permutation
\begin{equation*}
\chi_{n,1}=(1,2,\dots,n,n+1).
\end{equation*}
Using only the structure maps of $\EE$ would not give a map of motivic symmetric spectra.
The map (\ref{naturalshiftmap}) is not a stable weak equivalence in general. 
\begin{definition}
A motivic symmetric spectrum $\EE$ is called {\em semistable\/} if (\ref{naturalshiftmap}) is a 
stable weak equivalence of underlying (non-symmetric) motivic spectra.
\end{definition}

\begin{proposition}
\label{prop:semistable}
Let $\EE$ be a motivic symmetric spectrum such that for every $n$ and every permutation $\sigma\in \Sigma_n$ with sign 
$\sign(\sigma)=1$ the action of $\sigma$ on $\EE_n$ coincides with the identity in the pointed motivic unstable homotopy category. 
Then $\EE$ is semistable.
\end{proposition}
\begin{proof}
We may assume $\EE$ is levelwise fibrant.
Then the standard natural stabilization construction $Q$ gives a stably fibrant replacement of $\EE$.
Recall that in level $n$
\begin{equation}
\label{eq:q}
Q(\EE)_n 
= 
\colim_{k}  
(\EE_n \to \Omega^{2,1}\EE_{n+1} \to \dotsm \to \Omega^{2k,k}\EE_{n+k} \to \dotsm), 
\end{equation}
where the colimit is taken over the structure maps. 
It suffices to show that $Q(\phi(\EE))$ is a levelwise weak equivalence.
The assumption on $\EE$ implies the composite map
\[ 
\Omega^{2k,k}F(S^{2,1}\smash \EE_{n+k}) 
\xrightarrow{\Omega^{2k,k}F(\phi(\EE)_n) } 
\Omega^{2k,k}F(\sh(\EE)_{n+k}) 
\xrightarrow{\mathrm{can}}  
\Omega^{2k+4,k+2}F(S^{2,1}\smash \sh(\EE)_{n+k+1})
\]
coincides with the canonical map 
\[
\Omega^{2k,k}F(S^{2,1}\smash \EE_{n+k}) 
\xrightarrow{\mathrm{can}} 
\Omega^{2k+4,k+2}F(S^{2,1}\smash \EE_{n+k+2})
\] 
in the pointed motivic unstable homotopy category. 
The canonical maps denoted in the above by `$\mathrm{can}$' appear implicitly in (\ref{eq:q}). 
Thus $\phi(\EE)$ induces a weak equivalence on colimits 
\[ 
Q(F(S^{2,1}\smash \EE))_n  
\to 
Q(F(\sh(\EE)))_n 
\]
for every $n$.
\end{proof}

\begin{example}
\label{ex:semistable}
The motivic symmetric spectrum $\Sigma^\infty {\BG}_+$ is semistable.
This follows from Proposition \ref{prop:semistable} because the even permutations are homotopic to the identity map on the motivic spheres. 
For the same reason, the motivic symmetric spectrum $\KGL^{\beta}$ is semistable.
\end{example}

\begin{theorem}
\label{thm:semistable}
Let $\EE$ be a semistable motivic symmetric spectrum, 
and let $U$ denote the right Quillen functor to motivic spectra that forgets the symmetric group actions.
Then the value of the total right derived functor of $U$ at $\EE$ is $U(\EE)$.
\end{theorem}
\begin{proof}
We may assume $\EE$ is cofibrant and levelwise fibrant.
Let $R^\infty \EE$ denote the colimit of the sequence
\[  
\EE 
\xrightarrow{\phi(\EE)^\bigstar} 
\Omega^{2,1}(\sh(\EE) )
\xrightarrow{\Omega^{2,1}(\sh(\phi(\EE)^\bigstar))} 
\dotsm 
\]
in the category of motivic symmetric spectra.
Here $\phi(\EE)^\bigstar$ is the adjoint of the map $\phi(\EE)$ defined in (\ref{naturalshiftmap}). 
By assumption $U(\phi(\EE))$ is a stable weak equivalence of motivic spectra. 
Since $U$ commutes with the functors $S^{2,1}\smash -$ and $\Omega^{2,1}$ 
forming a Quillen equivalence
on motivic spectra, the derived adjoint
\[ U(\EE) \to  \Omega^{2,1}( U(\sh (\EE)) )\to \Omega^{2,1}((U\sh(\EE))^\mathrm{fib}) \]
is a stable weak equivalence. The stably fibrant replacement functor $Q$
commutes with $\Omega^{2,1}$ due to the finiteness of
$S^{2,1}$. Thus the map 
\[ \Omega^{2,1} (U(\sh (\EE))) \to \Omega^{2,1}Q(U(\sh(\EE))) \]
is a stable weak equivalence.
It follows that $U(\phi(\EE)^\bigstar)$ is also a stable weak equivalence of motivic spectra.

Hence the canonical map $\EE\to R^\infty \EE$ is a stable weak equivalence of underlying motivic spectra. 
By \cite[Theorem 18]{RO2} the same map is also a stable weak equivalence of motivic symmetric spectra. 
Since $R^\infty \EE$ is stably fibrant in the category of motivic symmetric spectra by construction, 
it gives a fibrant replacement of $\EE$. 
The result follows now, 
since
\[ 
U(\EE) 
\to 
U(R^\infty \EE) 
\]
is a stable weak equivalence.
\end{proof}

Define 
\begin{equation*}
a
\colon 
F(\Sigma^\infty {\BG}_+ )
\to 
\Omega^{4,2} \sh(F(\Sigma^\infty {\BG}_+))
\end{equation*}
to be the adjoint of the map
\begin{equation*}
S^{4,2}\smash F(\Sigma^\infty {\BG}_+ )
\to
\sh(F(\Sigma^\infty {\BG}_+ )).
\end{equation*}
In level $n$ the latter is the composite map 
\[ 
S^{4,2}\smash F(\Sigma^\infty {\BG}_+ )_n 
\xrightarrow{b\smash \id} 
F(\Sigma^\infty {\BG}_+ )_1\smash F(\Sigma^\infty {\BG}_+ )_n 
\xrightarrow{\mu_{1,n}} F(\Sigma^\infty {\BG}_+ )_{1+n}. 
\]

\begin{corollary}
\label{cor:semistable}
In the motivic stable homotopy category there exists an isomorphism between the homotopy colimit of the 
diagram of motivic symmetric spectra
\[ 
F(\Sigma^\infty {\BG}_+) 
\xrightarrow{a} 
\Omega^{4,2} \sh(F(\Sigma^\infty {\BG}_+)) 
\xrightarrow{\Omega^{4,2}\sh(a)} 
\dotsm 
\]
and the homotopy colimit $\Sigma^\infty {\BG}_+[\beta^{-1}]$ of the Bott tower
\begin{equation}
\label{equation:botttower}
\Sigma^\infty {\BG}_+ 
\xrightarrow{\beta} 
\Sigma^{-2,-1}\Sigma^\infty {\BG}_+
\xrightarrow{\Sigma^{-2,-1}\beta} 
\dotsm. 
\end{equation}
\end{corollary}
\begin{proof}
Due to semistability of $F(\Sigma^\infty {\BG}_+ )$, 
established in Example~\ref{ex:semistable}, 
we may identify $\Omega^{4,2}\sh(F(\Sigma^\infty {\BG}_+ ))$ with
$\Omega^{4,2}F(S^{2,1}\smash F(\Sigma^\infty {\BG}_+ ))$ and thus with $\Omega^{2,1}F(\Sigma^\infty {\BG}_+ )$ 
up to stable weak equivalence. 
The result follows since $a$ lifts the multiplication by the Bott element map (by construction).
\end{proof}

\begin{theorem}
\label{thm:stable-eq}
The motivic symmetric spectrum $\KGL^{\beta}$ has the homotopy type of the Bott inverted motivic spectrum $\Sigma^\infty {\BG}_+[\beta^{-1}]$.
\end{theorem}
\begin{proof}
Corollary~\ref{cor:semistable} identifies the Bott inverted motivic spectrum $\Sigma^\infty {\BG}_+[\beta^{-1}]$ with the homotopy colimit of the diagram
\begin{equation}
\label{homotopycolimit}
F(\Sigma^\infty {\BG}_+) 
\xrightarrow{a} 
\Omega^{4,2} \sh(F(\Sigma^\infty {\BG}_+)) 
\xrightarrow{\Omega^{4,2}\sh(a)} 
\dotsm. 
\end{equation}
Since the loop and shift functors appearing in (\ref{homotopycolimit}) preserve semistability,
it follows that the terms are semistable.
Next we shall identify the homotopy colimit of (\ref{homotopycolimit}) with $\KGL^{\beta}$.
In effect, 
note that leaving the symmetric groups actions aside, 
$\KGL^{\beta}$ is the diagonal of the diagram of motivic symmetric spectra in (\ref{homotopycolimit}).
Example \ref{ex:semistable} and Theorem \ref{thm:semistable} show that the value of the right derived functor of $U$ at $\KGL^{\beta}$ is given 
by forgetting the group actions on $\KGL^{\beta}$. 
Hence there exists an abstract isomorphism between $\KGL^{\beta}$ and $\Sigma^\infty {\BG}_+[\beta^{-1}]$ in the motivic stable homotopy category.
\end{proof}

Next we discuss in broad strokes a motivic functor model for $K$-theory.
There exists a strict symmetric monoidal functor from motivic symmetric spectra to motivic functors 
\begin{equation*}
\mathbf{MSS}
\to
\mathbf{MF}
\end{equation*}
for the base scheme $S$, which is part of a Quillen equivalence proven in \cite{DROstable}.
Therefore the image $\KGL^{\beta}$ of the motivic symmetric spectrum model for $K$-theory yields a strict 
motivic functor model for $K$-theory.
In order to make this model more explicit, 
one could try to construct it entirely within the framework of motivic functors by starting out with the motivic functor 
${\BG}_+\smash -$ and go through constructions reminiscent of the ones for motivic symmetric spectra in this paper.
\vspace{0.2in}

\begin{remark}
Applying the arguments in this paper to Kan's (lax symmetric monoidal) fibrant replacement functor for simplicial sets and the Bott element in 
$\pi_{2}\Sigma^{\infty}\B\CC^{\times}$ yields a commutative symmetric ring spectrum with the homotopy type of topological unitary $K$-theory.
More generally, 
for $A$ an abelian compact Lie group,
the same argument applies to the Bott inverted model for $A$-equivariant unitary topological ${K}$-theory in \cite{SO2}.
We leave further details to the interested reader.
\end{remark}

\section{Multiplicative structure}
\label{section:multiplicativestructure}
\begin{theorem}
The multiplicative structures on $\KGL^{\beta}$ and $\KGL$ coincide in the motivic stable homotopy category.
\end{theorem}
\begin{proof}
The proof proceeds by showing there is a commutative diagram of monoids
\begin{equation}
\label{equation:commutativediagrams}
\xymatrix{ 
\KGL^{\beta} \ar[r] & \Sigma^{\infty}{\BG}_{+} [\beta^{-1}]\ar[d]  \\
\Sigma^{\infty} {\BG}_{+} \ar[u] \ar[ur] \ar[r] &  \KGL }
\end{equation}
in the motivic stable homotopy category.
\vspace{0.1in}

On the level of bigraded homology theories there is a commutative diagram:
\begin{equation*}
\xymatrix{ 
& 
(\Sigma^{\infty}{\BG}_{+} [\beta^{-1}])_{\ast,\ast}(\,\,)  \ar[d] \\
(\Sigma^{\infty}{\BG}_{+})_{\ast,\ast}(\,\,)   \ar[ur] \ar[r]  & \KGL_{\ast,\ast}(\,\,)    }
\end{equation*}
This diagram lifts uniquely to a commutative diagram of monoids in the motivic stable homotopy category, 
as asserted by the right hand side of (\ref{equation:commutativediagrams}), 
since for $K$-theory there exist no nontrivial phantom maps according to  \cite[Remark 9.8 (ii), (iv)]{NSO}.
\vspace{0.1in}

On the left hand side of (\ref{equation:commutativediagrams}), 
recall $\Sigma^{\infty} {\BG}_{+}\rightarrow\KGL^{\beta}$ is a map of motivic symmetric ring spectra.
For the discussion of $\KGL^\beta \to \Sigma^{\infty}{\BG}_{+} [\beta^{-1}]$ we shall use the following model for the homotopy colimit.
Let $\EE$ be the stably fibrant replacement of $\Sigma^{\infty}{\BG}_{+} $ obtained by first applying the functor $F$ levelwise and second the stabilization functor $Q$.
Now define $\EE[\beta^{-1}]$ as the diagonal spectrum of the naturally induced sequence
\[ 
\EE 
\to 
\Sigma^{-2,-1} \EE 
\to 
\dotsm 
\]
lifting the Bott tower (\ref{equation:botttower}).
Here $\Sigma^{-2,-1} \EE$ is realized as a shift, 
so that $\EE[\beta^{-1}]_n=\EE_0$ and its structure maps are given by multiplication with the Bott element. 
In level $n$ the map $\KGL^\beta \to\EE[\beta^{-1}]$ is the canonical map 
\[
\Omega^{4n,2n}F({\BG}_+\smash S^{4n,2n}) 
\to 
\Omega^{2\infty,\infty}F({\BG}_+\smash S^{2\infty,\infty}) 
= 
\EE_0. 
\]
When $n=0$ the latter map corresponds via adjointness to the diagonal map in (\ref{equation:commutativediagrams}). 
The evident monoid structure on ${\BG}_+$ induces a monoid structure on $\EE_0$ and hence a naive multiplication on $\EE[\beta^{-1}]$ given by 
\[ 
\EE[\beta^{-1}]_m \smash\EE[\beta^{-1}]_n 
= 
\EE_0\smash\EE_0 
\to 
\EE_0 
= 
\EE[\beta^{-1}]_{m+n}. 
\]
Now from the construction of the ring structure on $\KGL^\beta$ it follows that $\KGL^\beta \to\EE[\beta^{-1}]$ respects  the naive product.
\end{proof}

\vspace{0.09in}

\begin{center}
Institut f{\"u}r Mathematik, Universit{\"a}t Osnabr{\"u}ck, Germany. \\
e-mail: oroendig@math.uos.de
\end{center}
\begin{center}
Fakult{\"a}t f{\"u}r Mathematik, Universit{\"a}t Regensburg, Germany.\\
e-mail: Markus.Spitzweck@mathematik.uni-regensburg.de
\end{center}
\begin{center}
Department of Mathematics, University of Oslo, Norway.\\
e-mail: paularne@math.uio.no
\end{center}
\end{document}